\documentclass[twoside,leqno,10pt]{amsart}
\usepackage{amsfonts}
\usepackage{amsmath}
\usepackage{amscd}
\usepackage{amssymb}
\usepackage{amsthm}
\usepackage{amsrefs}
\usepackage{latexsym}
\usepackage{mathrsfs}
\usepackage{bbm}
\usepackage{enumerate}
\usepackage{color}
\begin{document}

\newtheorem{theorem}[subsection]{Theorem}
\newtheorem{proposition}[subsection]{Proposition}
\newtheorem{lemma}[subsection]{Lemma}
\newtheorem{corollary}[subsection]{Corollary}
\newtheorem{conjecture}[subsection]{Conjecture}
\newtheorem{prop}[subsection]{Proposition}
\numberwithin{equation}{section}
\newcommand{\mr}{\ensuremath{\mathbb R}}
\newcommand{\mc}{\ensuremath{\mathbb C}}
\newcommand{\dif}{\mathrm{d}}
\newcommand{\intz}{\mathbb{Z}}
\newcommand{\ratq}{\mathbb{Q}}
\newcommand{\natn}{\mathbb{N}}
\newcommand{\comc}{\mathbb{C}}
\newcommand{\rear}{\mathbb{R}}
\newcommand{\prip}{\mathbb{P}}
\newcommand{\uph}{\mathbb{H}}
\newcommand{\fief}{\mathbb{F}}
\newcommand{\majorarc}{\mathfrak{M}}
\newcommand{\minorarc}{\mathfrak{m}}
\newcommand{\sings}{\mathfrak{S}}
\newcommand{\fA}{\ensuremath{\mathfrak A}}
\newcommand{\mn}{\ensuremath{\mathbb N}}
\newcommand{\mq}{\ensuremath{\mathbb Q}}
\newcommand{\half}{\tfrac{1}{2}}
\newcommand{\f}{f\times \chi}
\newcommand{\summ}{\mathop{{\sum}^{\star}}}
\newcommand{\chiq}{\chi \bmod q}
\newcommand{\chidb}{\chi \bmod db}
\newcommand{\chid}{\chi \bmod d}
\newcommand{\sym}{\text{sym}^2}
\newcommand{\hhalf}{\tfrac{1}{2}}
\newcommand{\sumstar}{\sideset{}{^*}\sum}
\newcommand{\sumprime}{\sideset{}{'}\sum}
\newcommand{\sumprimeprime}{\sideset{}{''}\sum}
\newcommand{\shortmod}{\ensuremath{\negthickspace \negthickspace \negthickspace \pmod}}
\newcommand{\V}{V\left(\frac{nm}{q^2}\right)}
\newcommand{\sumi}{\mathop{{\sum}^{\dagger}}}
\newcommand{\mz}{\ensuremath{\mathbb Z}}
\newcommand{\leg}[2]{\left(\frac{#1}{#2}\right)}
\newcommand{\muK}{\mu_{\omega}}
\newcommand{\sumflat}{\sideset{}{^\flat}\sum}

\title[Average values of quadratic Hecke character sums]{Average values of quadratic {H}ecke character sums}

%%\date{\today}
\author{Peng Gao and Liangyi Zhao}

\begin{abstract}
 We study certain smoothed character sums involving $\sum_{m,n}\leg {m}{n}_2$, where $\leg {\cdot}{n}_2$ denotes the quadratic symbol in the Gaussian field.
  We extend some previously known results to obtain asymptotic formulas for the sums considered to larger ranges of $m$ and $n$.
\end{abstract}

\maketitle


\section{Introduction}

 An important theme in analytic number theory is the
estimations of character sums. The classical P\'olya-Vinogradov inequality (see for example \cite[Chap. 23]{Da}) states
that for any non-principal Dirichlet character $\chi$ modulo $q$ and $M \in \mz$, $N \in \mn$, we have
\begin{align*}
%%\label{pv}
   \sum_{M < n \leq M+N} \chi(n) \ll q^{1/2}\log q.
\end{align*}
The above estimate has many interesting applications including the study of least quadratic non-residues, least primitive roots and primes in
arithmetic progressions. \newline

  Another important topic studied in analytic number theory is the asymptotic behaviors of mean values. In view of the P\'olya-Vinogradov inequality,
we see that it is not possible to obtain an asymptotic formula for a single character sum. In order to obtain such formulas, we need to consider extra averagings.  \newline

  In the case of quadratic Dirichlet character sums, we are led to consider the following sum:
\begin{align}
\label{S2}
  S(X,Y)=\sum_{\substack {m \leq X \\ (m, 2)=1}}\sum_{\substack {n \leq Y \\ (n, 2)=1}} \leg {m}{n}.
\end{align}

   Although it is relatively easy to obtain an asymptotic formula of $S(X, Y)$ if $Y=o(X/\log X)$ or $X=o(Y/\log Y)$ with the help of
the P\'olya-Vinogradov inequality, the situation becomes much more subtle and delicate if $X$ and $Y$ are of comparable size.
Using a Poisson summation formula developed in \cite{sound1}, J. B. Conrey, D. W. Farmer and K. Soundararajan \cite{CFS} succeeded
in obtaining an asymptotic formula of $S(X, Y)$ for all $X$ and $Y$. \newline

Thus one is motiviated by the result of Conrey, Farmer and Soundararajan to study similar character sums.   For example, let $K=\mq(i)$ be the Gaussian field and $\mathcal{O}_K=\mz[i]$ the ring of integers in $K$. For every $c \in \mathcal{O}_K$ such that $(c, 1+i)=1$, let $\leg {\cdot}{c}_2$ be the quadratic
residue symbols defined in Section \ref{sec2.4}.  Suppose that $\Phi(t), W(t)$ are two non-negative smooth functions compactly supported on $\mr^+$, we define
\begin{align}
\label{SXY}
  S_2(X,Y;\Phi, W) =\sum_{n \equiv 1 \bmod {(1+i)^3}}\sum_{(m, 1+i)=1} \leg {(1+i)m}{n}_2\Phi \left(\frac {N(n)}{Y} \right) W \left( \frac {N(m)}{X} \right).
\end{align}

   A similar expression without the appearance of $1+i$ in the quadratic symbol above has been studied previously by the authors in \cite{G&Zhao2019}. An asymptotic formula is obtained for such an expression, which is valid when $Y = o(X)$ and $X = o(Y^{2-\varepsilon})$ for any $\varepsilon>0$. \newline

  Our goal in this paper is to further continue our investigation on $S_2(X,Y; \Phi, W)$, aiming at obtaining a valid asymptotic formula for all large $Y$ going up to $X$.
The new input here is the method used by  K. Soundararajan and M. P. Young in \cite{S&Y}.
After using Poisson summation first to transform one of the sums in $S_2$ into a dual sum, we shall follow the approach in \cite{S&Y} to treat
the resulting sums using multiple Dirichlet series. This allows us to better control the emerging error terms.  \newline

 Before stating the result, we note that, as pointed out in \cite{G&Zhao2019}, one can regard $\leg {m}{\cdot}_2$ as a Hecke character $\chi_m \pmod {16m}$
of trivial infinite type when $(m, 1+i)=1$, thus justifying our view of $S_2$ as Hecke character sums.  Moreover, the factor $\leg {1+i}{n}_2$ is inserted in the expression of $S_2$ for a purely technical reason.  One can obtain a similar asymptotic formula for $S_2$ without this factor.  \newline

  Our main result is as follows.
\begin{theorem}
\label{quadraticmean}
   Let $\Phi$ and $W$ be two non-negative, smooth, compactly supported functions on $\mr^{+}$. Let $S_2(X,Y;\Phi, W)$ be defined in \eqref{SXY}.
For large $X$ and  $Y$ with $X \geq Y$, $\theta=131/416$, we have for any $\varepsilon>0$,
\begin{align}
\label{S}
\begin{split}
   S_2(X,Y;\Phi, W) =& C(\Phi, W)XY^{1/2} +O \left (XY^{\theta/2}+ XY^{1/2} \left( \frac {Y}{X} \right)^{1+\varepsilon} \right),
\end{split}
\end{align}
  where $C(\Phi, W)$ is given in \eqref{C}.
\end{theorem}

   We remark here that the result of Conrey, Farmer and Soundararajan on $S(X,Y)$ defined in \eqref{S2} is that
\begin{align*}
   S(X,Y)= \frac 2{\pi^2}C \left(\frac {Y}{X} \right)X^{3/2}+O\left( \left( XY^{7/16}+YX^{7/16} \right) \log XY \right).
\end{align*}

   Moreover, it was shown in \cite{CFS} that
\begin{align*}
    C(\alpha) =\sqrt{\alpha}+O\left( \alpha^{3/2} \right), \quad \alpha \rightarrow 0.
\end{align*}

    It follows that when $Y < X^{1-\varepsilon}$, we can recast $S(X,Y)$ as
\begin{align*}
   S(X,Y)= \frac 2{\pi^2}XY^{1/2}+O\left( Y^{3/2} \right).
\end{align*}

    The error term given in Theorem \eqref{quadraticmean} is essentially of the same order of magnitude as that in the above expression. Moreover,
the asymptotic formula given in \eqref{S} is now valid for all $Y \leq X^{1-\varepsilon}$.

\section{Preliminaries}
\label{sec 2}
%%----------------------------------------------------------------------------
\subsection{Quadratic and quartic symbols}
\label{sec2.4}
%%----------------------------------------------------------------------------
    The symbol $\leg{\cdot}{n}_4$ is the quartic
residue symbol in the ring $\mz[i]$.  For a prime $\varpi \in \mz[i]$
with $N(\varpi) \neq 2$, the quartic symbol is defined for $a \in
\mz[i]$, $(a, \varpi)=1$ by $\leg{a}{\varpi}_4 \equiv
a^{(N(\varpi)-1)/4} \pmod{\varpi}$, with $\leg{a}{\varpi}_4 \in \{
\pm 1, \pm i \}$. When $\varpi | a$, we define
$\leg{a}{\varpi}_4 =0$.  Then the quartic character can be extended
to any composite $n$ with $(N(n), 2)=1$ multiplicatively. We extend the definition of $\leg{\cdot }{n}_4$ to $n=1$ by setting $\leg{\cdot}{1}_4=1$.
We further define $(\frac{\cdot}{n})=\leg {\cdot}{n}^2_4$  to be the quadratic
residue symbol for these $n$'s.  \newline

 Note that in $\intz[i]$, every ideal coprime to $2$ has a unique
generator congruent to 1 modulo $(1+i)^3$.  Such a generator is
called primary.

\subsection{Gauss sums}
\label{section:Gauss}
For any $n, r\in \mz[i]$, $n \equiv 1 \pmod {(1+i)^3}$, the quadratic
 Gauss sum $g_2(r, n)$ is defined by
\begin{align}
\label{g2}
 g_2(r,n) = \sum_{x \bmod{n}} \leg{x}{n} \widetilde{e}_i\leg{rx}{n}, \; \mbox{where} \;\widetilde{e}_i(z) =\exp \left( 2\pi i  \left( \frac {z}{2i} - \frac {\bar{z}}{2i} \right) \right) .
\end{align}

The following property of $g_2(r,n)$  can be easily derived from its definition.
\begin{align}
\label{eq:gmult}
 g_2(rs,n) & = \overline{\leg{s}{n}}_2 g_2(r,n), \quad (s,n)=1, \qquad \mbox{$n$ primary}.
\end{align}

   We write $N(n)$ for the norm of any $n \in \mz[i]$ and $\varphi(n)$ for the number of elements in the reduced residue class of $\mathcal{O}_K/(n)$.
The next lemma allows us to evaluate $g_2(r,n)$ for $n \equiv 1 \pmod {(1+i)^3}$ explicitly.
\begin{lemma}{\cite[Lemma 2.4]{G&Zhao4}}
\label{Gausssum}
\begin{enumerate}[(i)]
\item  For $m,n$ primary and $(m , n)=1$, we have
\begin{align*}
%%\label{2.7}
   g_2(k,mn) & = g_2(k,m)g_2(k,n).
\end{align*}
\item Let $\varpi$ be a primary prime in $\mz[i]$. Suppose $\varpi^{h}$ is the largest power of $\varpi$ dividing $k$. (If $k = 0$ then set $h = \infty$.) Then for $l \geq 1$,
\begin{align*}
g_2(k, \varpi^l)& =\begin{cases}
    0 \qquad & \text{if} \qquad l \leq h \qquad \text{is odd},\\
    \varphi(\varpi^l)=\#(\mz[i]/(\varpi^l))^* \qquad & \text{if} \qquad l \leq h \qquad \text{is even},\\
    -N(\varpi)^{l-1} & \text{if} \qquad l= h+1 \qquad \text{is even},\\
    \leg {ik\varpi^{-h}}{\varpi}N(\varpi)^{l-1/2} \qquad & \text{if} \qquad l= h+1 \qquad \text{is odd},\\
    0, \qquad & \text{if} \qquad l \geq h+2.
\end{cases}
\end{align*}
\end{enumerate}
\end{lemma}

    As an immediate consequence of the above lemma, we see that for any $k, n \in \mz[i]$ with $n$ primary,
\begin{align}
 \label{gbound}
   g_2(k,n) \ll N(n).
\end{align}

\subsection{Poisson summation}
  The proof of Theorems \ref{quadraticmean} requires the following Poisson summation formula (see \cite{G&Zhao4} for details):
\begin{lemma}{\cite[Corollary 2.12]{G&Zhao4}}
\label{Poissonsumformodd} Let $n \in \mz[i],   n \equiv 1 \pmod {(1+i)^3}$ and $\leg {\cdot}{n}_2$  be the quadratic  residue symbol $\pmod {n}$. For any Schwartz class function $W$,  we have
\begin{align*}
   \sum_{\substack {m \in \mz[i] \\ (m,1+i)=1}}\leg {m}{n}_2 W\left(\frac {N(m)}{X}\right)=\frac {X}{2N(n)}\leg {1+i}{n}_2\sum_{k \in
   \mz[i]}(-1)^{N(k)} g_2(k,n)\widetilde{W}_i\left(\sqrt{\frac {N(k)X}{2N(n)}}\right).
\end{align*}
   where
\begin{align*}
   \widetilde{W}_i(t) &=\int\limits^{\infty}_{-\infty}\int\limits^{\infty}_{-\infty}W(N(x+yi))\widetilde{e}_i\left(- t(x+yi)\right)\dif x \dif y, \quad t \geq 0.
\end{align*}
\end{lemma}

\subsection{Evaluation of $\widetilde{W}_i(t)$}
\label{sect: Wiestm}
     We will require some simple estimates on $\widetilde{W}_i(t)$ and its derivatives. First note that for any $t \geq 0$, $\widetilde{W}_i(t) \in \mr$ since
\begin{align*}
%%\label{Wt}
     \widetilde{W}_i(t)  =\int\limits_{\mr^2}\cos (2\pi t y)W(x^2+y^2) \ \dif x \dif y.
\end{align*}

Changing to polar coordinates, we get
\begin{align*}
     \widetilde{W}_i(t) =& 4\int\limits^{\pi/2}_0\int\limits^{\infty}_0\cos (2\pi t r\sin \theta)W(r^2) \ r \dif r \dif \theta    = 2\int\limits^{\pi/2}_0 \int\limits^{\infty}_0\cos (2\pi t r^{1/2}\sin \theta)W(r) \ \dif r \dif \theta.
\end{align*}

Now apply Mellin inversion yields
\begin{align*}
     \widetilde{W}_i(t) =& 2\int\limits^{\pi/2}_0\int\limits^{\infty}_0\cos (2\pi t r^{1/2}\sin \theta)\frac 1{2\pi i}\int\limits_{(-1/2)}\widehat{W} \left( 1+\frac s2 \right)r^{-s/2}\dif s \frac {\dif r}{r} \dif \theta.
\end{align*}
Here and after, for any function $f$, its Mellin transform, $\hat{f}$, is defined to be
\begin{align*}
     \widehat{f}(s) =\int\limits^{\infty}_0f(t)t^s\frac {\dif t}{t}.
\end{align*}

  We make some changes of variables (first $r^{1/2} \to r$,
then $2 \pi tr \sin \theta \to r$ and $s \to -s$) to get that
\begin{align}
\label{FnFT}
\begin{split}
  \widetilde{W}(t)  =& \frac 4{2\pi i}
\int\limits\limits_{(1/2)}\widehat{W}(1-\frac s2)\int\limits^{\pi/2}_0
\int\limits^{\infty}_0\cos (r)\left(\frac r{2\pi t \sin \theta}\right )^{s} \frac {\dif r}{r} \dif \theta \dif s    \\
     =&
 \frac {4}{2\pi i}
 \int\limits\limits_{(1/2)}\widehat{W}(1-\frac s2)
(2\pi t)^{-s}\left ( \int\limits^{\pi/2}_0 (\sin \theta )^{-s} \dif \theta
\int\limits^{\infty}_0\cos (r)r^{s}\frac {\dif r}{r} \right ) \dif s \\
=&
 \frac {\pi}{2\pi i}
 \int\limits\limits_{(1/2)}\widehat{W}(1-\frac s2)
(\pi t)^{-s}\frac{\Gamma (s/2)}{\Gamma ((1-s)/2)} \dif s,
\end{split}
\end{align}
  where the last line above follows from the relation (see \cite[Section 2.4]{Gao2002}) that
\begin{align*}
\int\limits^{\pi/2}_0 (\sin \theta )^{-u} \dif \theta
\int\limits^{\infty}_0\cos (r)r^{u}\frac {\dif r}{r}=\frac {\pi}{2}2^{u-1}\frac{\Gamma (s/2)}{\Gamma ((1-s)/2)} .
\end{align*}

\subsection{Analytical behavior of a Dirichlet series}
\label{sect: alybehv}

   Let $g_2(k,n)$ be defined as in \eqref{g2}. We now fix a generator for every prime ideal $(\varpi) \subset \mathcal{O}_K$ by taking $\varpi$ to be primary if $(\varpi, 1+i)=1$ and $1+i$ for the ideal $(1+i)$ (noting that $(1+i)$ is the only prime ideal in $\mathcal{O}_K$ that lies above the integral ideal $(2) \subset \mz$). We also fix  $1$ as the generator for the ring $\mz[i]$ itself and
extend the choice of the generator for any ideal of $\mathcal{O}_K$ multiplicatively. We denote the set of such generators by $G$. \newline

   We say any element $k \in \mathcal{O}_K$ is square-free if the ideal $(k)$ is not divisible by the square of any prime ideal. Now, for any square-free $k_1 \in \mathcal{O}_K$, we define
\begin{align}
\label{eq:J}
  J_{k_1}(v,w) &=\sum_{n \equiv 1 \bmod {(1+i)^3}}  \sum_{\substack {k_2}}\frac {1}{N(n)^{w}N(k_2)^{2v}} \frac {g_2(k_1k^2_2,n)}{N(n)},
\end{align}
    where we use the convention throughout that all sums over $k_2$ are restricted to $k_2 \in G$. \newline

  It follows from \eqref{gbound} that $g_2(k_1k^2_2,n) \ll N(n)$, so that $J_{k_1}(v,w)$ converges absolutely if $\Re(w)$ and $\Re(v)$ are
both strictly greater than $1$.  For any Hecke character $\chi$, we write $L(s,\chi)$ for the associated $L$-function. The next lemma describes certain analytical behavior of $J_{k_1}(v,w)$.
\begin{lemma}
\label{lemma:J}
 The function $J_{k_1}(v,w)$ defined above may be written as
 \begin{align*}
 L\left( \frac 12+w, \chi_{i k_1} \right)J_{2, k_1}(v,w),
\end{align*}
 where $J_{2, k_1}(v,w)$ is a function uniformly bounded by $N(k_1)^{\varepsilon}$ for any $\varepsilon>0$ in the region
$\Re(v) \ge 1/2+\varepsilon$, and $\Re(w) \geq \varepsilon$.
\end{lemma}
\begin{proof}
It follows from Lemma \ref{Gausssum} that the summands on the right-hand side of \eqref{eq:J} are jointly multiplicative in terms of $n$ and $k_2$, so that $J_{k_1}(v,w)$ can be
expressed as an Euler product over all primary primes $\varpi$ with each factor at $\varpi$ being
\begin{align*}
 J_{\varpi,k_1}(v,w):=\sum_{k_2 \geq 0, n \geq 0} \frac{1}{N(\varpi)^{n w +2k_2 v}}\frac{ g_2(k_1 \varpi^{2k_2}, \varpi^{n})}{N(\varpi)^{n}  }.
\end{align*}
Applying \eqref{gbound} again, we have $J_{\varpi,k_1}(v,w) \ll 1$ uniformly for all $\varpi$. \newline

 For $\varpi \nmid k_1$, we can replace $g_2(k_1 \varpi^{2k_2}, \varpi^{n})$ in the definition of $J_{\varpi,k_1}(v,w)$ by
 an explicit evaluation of it using Lemma \ref{Gausssum}. Keeping only the non-zero terms, this yields an alternative expression for
  $J_{\varpi,k_1}(v,w)$, which we denote by
 $J^{gen}_{\varpi,k_1}(v,w)$. One checks that we have
\begin{equation*}
 J^{gen}_{\varpi,k_1}(v,w):=\sum_{k_2 \geq 0} \left ( \sum_{\substack{ n \geq 0 \\ n \equiv 0 \pmod 2 \\ n \leq 2k_2}}
 \frac{1}{N(\varpi)^{nw +2k_2 v}}\frac{ \varphi(\varpi^{n})}{N(\varpi)^{n}}+\sum_{\substack{ n \geq 0 \\ n= 2k_2+1}}
 \frac{1}{N(\varpi)^{nw +2k_2 v}}\frac{\chi_{ik_1}(\varpi)}{N(\varpi)^{1/2}  }  \right ).
\end{equation*}
 We now extend the above definition of $J^{gen}_{\varpi,k_1}(v,w)$ to all other $\varpi$ (including $\varpi=1+i$). \newline

In the region $\Re(v) \ge 1/2+\varepsilon$, $\Re(w)\ge \varepsilon$, it follows from
the definition of $J^{gen}_{\varpi,k_1}(v,w)$ that the contribution from terms $k_2\ge 1$ is $\ll 1/N(\varpi)^{1+2\varepsilon}$. The contribution of the term $k_2=0$ is $1+ \chi_{i k_1}(\varpi) N(\varpi)^{-1/2-w}$. \newline

  We now define
\begin{align*}
  J_{2, k_1}(v,w) =& \left ( L \left( \frac 12+w, \chi_{i k_1} \right)\right )^{-1}J_{k_1}(v,w)
 = J^{gen}_{2,k_1}(v,w)J^{non-gen}_{2, k_1}(v,w),
\end{align*}
  where
\[
  J^{gen}_{2,k_1}(v,w)= \prod_{\varpi}\left(1-\frac {\chi_{i k_1}(\varpi)}{N(\varpi)^{1/2+w}} \right )J^{gen}_{\varpi,k_1}(v,w), \quad
  J^{non-gen}_{2, k_1}(v,w)= \prod_{\varpi | (1+i)k_1}J^{gen}_{\varpi,k_1}(v,w)^{-1}J_{\varpi,k_1}(v,w),
\]
and we define $J_{\varpi,k_1}(v,w)=1$ when $\varpi=1+i$. \newline

  Our arguments above imply that $J^{gen}_{2,k_1}(v,w)$ is uniformly bounded by $1$ in the region
$\Re(v) \ge 1/2+\varepsilon$, and $\Re(w) \geq \varepsilon$. As one checks easily that $J^{non-gen}_{2, k_1}(v,w)$
is uniformly bounded by $N(k_1)^{\varepsilon}$ for any $\varepsilon>0$ in the same region, the assertions of the lemma now follow.
\end{proof}

\subsection{A mean value estimate}
   For $k \in \mathcal{O}_K$, recall that $\chi_{k}$ denote the quadratic symbol $\leg {k}{\cdot}_2$.
We shall use $\sum^*$ to denote a sum over square-free elements in $\mathcal{O}_K$ throughout the paper.
In the proof of Theorem \ref{quadraticmean}, we need the following mean value estimate for quadratic Hecke $L$-functions.
\begin{lemma}
\label{lem:HB}
   For any complex number $\sigma+it$  with $\sigma \geq 1/2$, we have
\begin{align}
\label{eq:Qbound}
 \sumstar_{\substack{N(k_1) \leq X}} |L(\sigma + it, \chi_{ik_1})|^2 \ll_{\varepsilon} (X(1+|t|))^{1 + \varepsilon},
\end{align}
   and that
\begin{align}
\label{eq:Lbound}
 \sumstar_{\substack{N(k_1) \leq X}} |L(\sigma + it, \chi_{ik_1})| \ll_{\varepsilon} (X(1+|t|)^{1/2})^{1 + \varepsilon}.
\end{align}
\end{lemma}
\begin{proof}
  The estimate \eqref{eq:Qbound} follows from \cite[Corollary 1.4]{BGL} and \eqref{eq:Lbound} follows from \eqref{eq:Qbound} by Cauchy's inequality.
\end{proof}

\section{Proof of Theorem \ref{quadraticmean}}
\label{sec 3}

    Applying the Poisson summation formula, Lemma \ref{Poissonsumformodd},
\begin{align*}
  S_2(X,Y; \Phi, W) &=\frac {X}{2}\sum_{k \in
   \mz[i]}(-1)^{N(k)} \sum_{n \equiv 1 \bmod {(1+i)^3}} \frac {g_2(k,n)}{N(n)} \Phi \left( \frac {N(n)}{Y} \right)\widetilde{W}_i\left(\sqrt{\frac {N(k)X}{2N(n)}}\right) \\
   & =\frac {X\widetilde{W}_i(0)}{2}\sum_{n \equiv 1 \bmod {(1+i)^3}} \frac {g_2(0,n)}{N(n)} \Phi \left( \frac {N(n)}{Y} \right) \\
   & \hspace*{2cm} +\frac {X}{2}\sum_{\substack {k \in
   \mz[i] \\ k \neq 0}}(-1)^{N(k)} \sum_{n \equiv 1 \bmod {(1+i)^3}}  \frac {g_2(k,n)}{N(n)} \Phi \left( \frac {N(n)}{Y} \right) \widetilde{W}_i\left(\sqrt{\frac {N(k)X}{2N(n)}}\right) \\
   &:= M_{0}+M', \quad \mbox{say}.
\end{align*}

  The term $M_{0}$ can be treated exactly as in \cite[Section 3.1]{G&Zhao2019} and the result there gives
\begin{align}
\label{M0}
  M_{0} = \frac {\pi XY^{1/2}}{24\zeta_{K}(2)}\widetilde{W}_i(0)\Phi \left( \frac 12 \right)+O \left( XY^{\theta/2} \right) .
\end{align}
     where $\zeta_K(s)$ is the Dedekind zeta function of $K$ and $\theta=131/416$ arises from the currently best known error term of the Gauss circle problem
     (see \cite{Huxley1}).

\subsection{The Term $M'$}
    Now suppose $k \neq 0$. Applying \eqref{FnFT} gives
\begin{align*}
    M' &=\frac {\pi}{2}X\sum_{\substack {k \in
   \mz[i] \\ k \neq 0}}(-1)^{N(k)} \sum_{n \equiv 1 \bmod {(1+i)^3}}  \frac {g_2(k,n)}{N(n)} \Phi \left( \frac {N(n)}{Y} \right)  \frac 1{2\pi i}\int\limits_{(1/2)}\widehat{W} \left( 1-\frac s2 \right)\left( \pi \sqrt{\frac {N(k)X}{2N(n)}}\right)^{-s} \frac{\Gamma (s/2)}{\Gamma ((1-s)/2)} \dif s .
\end{align*}

    We take the Mellin inversion of $\Phi$ to further recast $M'$ as
\begin{align*}
    M' &=\frac {\pi}{2}X \sum_{\substack {k \in
   \mz[i] \\ k \neq 0}}(-1)^{N(k)} \sum_{n \equiv 1 \bmod {(1+i)^3}}  \frac {g_2(k,n)}{N(n)} \\
   & \hspace*{2cm} \times \left( \frac 1{2\pi i} \right)^2\int\limits_{(1/2)}\int\limits_{(\varepsilon)}\widehat{\Phi}(u) \left( \frac {N(n)}{Y} \right)^{-u}\widehat{W} \left(1-\frac s2\right)\left(\pi \sqrt{\frac {N(k)X}{2N(n)}}\right)^{-s}\frac{\Gamma (s/2)}{\Gamma ((1-s)/2)}  \dif u \dif s \\
    &=\frac {\pi}{2}X\sum_{\substack {k \in
   \mz[i] \\ k \neq 0}}(-1)^{N(k)} \sum_{n \equiv 1 \bmod {(1+i)^3}}  \frac {g_2(k,n)}{N(n)} \\
   & \hspace*{2cm} \times \left( \frac 1{2\pi i} \right)^2\int\limits_{(3/2)}\int\limits_{(\varepsilon)}\widehat{\Phi}(u) \left( \frac {N(n)}{Y} \right)^{-u}\widehat{W} \left( 1-\frac s2 \right) \left(\pi \sqrt{\frac {N(k)X}{2N(n)}}\right)^{-s} \frac{\Gamma (s/2)}{\Gamma ((1-s)/2)} \dif u \dif s,
\end{align*}
   where the last equality above follows from the observation that the integral over $s$ may be taken over any vertical
lines with real part between $0$ and $2$. \newline

    Let $f(k)=g_2(k,n)/N(k)^s$ and we write $k = k_1k^2_2$ with $k_1$ square-free and $k_2 \in G$.  Recall here that $G$ is the set of generators of all ideals in $\mathcal{O}_K$ defined in Section \ref{sect: alybehv}. We break the sum over $k_1$ into two sums, depending on whether $k_1$ and $1+i$ are relatively prime or not, getting
\begin{align*}
   \sum_{\substack {k \in
   \mz[i] \\ k \neq 0}}(-1)^{N(k)} f(k)= & \sumstar_{\substack{k_1 \\ (k_1, 1+i) \neq 1}}\sum_{k_2}f(k_1k^2_2)+\sumstar_{\substack{k_1 \\ (k_1, 1+i) = 1}}\sum_{k_2}(-1)^{N(k_2)}f(k_1k^2_2).
\end{align*}

   Note that
\begin{align*}
   \sum_{k_2}(-1)^{N(k_2)}f(k_1k^2_2) =& \sum_{1+i|k_2}f(k_1k^2_2)-\sum_{(k_2, 1+i)=1}f(k_1k^2_2)
   = 2\sum_{1+i|k_2}f(k_1k^2_2)-\sum_{k_2}f(k_1k^2_2) \\
   =& 2\sum_{k_2}f(2ik_1k^2_2)-\sum_{k_2}f(k_1k^2_2). 
\end{align*}   
   
     It follows that we have
\begin{align*}
   \sum_{\substack {k \in
   \mz[i] \\ k \neq 0}}(-1)^{N(k)} f(k)=&  \sumstar_{\substack{k_1 \\ (k_1, 1+i) \neq 1}}\sum_{k_2}f(k_1k^2_2)+\sumstar_{\substack{k_1 \\ (k_1, 1+i) = 1}}
   \left (2\sum_{k_2}f(2ik_1k^2_2)-\sum_{k_2}f(k_1k^2_2) \right ) \\
   =&  \sumstar_{\substack{k_1 \\ (k_1, 1+i) \neq 1}}\sum_{k_2}f(k_1k^2_2)+\sumstar_{\substack{k_1 \\ (k_1, 1+i) = 1}}
   \left (2\sum_{k_2}f(2k_1k^2_2)-\sum_{k_2}f(k_1k^2_2) \right ),
\end{align*}     
   where the last equality above follows by a relabelling of $k_1$.  \newline

    Note that if $(n, 1+i)=1$, $g_2(k,n)=g_2(2k,n)$ by \eqref{eq:gmult}. It follows that we have $f(2k_1k^2_2)=4^{-s}f(k_1k^2_2)$ so that
\begin{align*}
   \sum_{\substack {k \in
   \mz[i] \\ k \neq 0}}(-1)^{N(k)} f(k)= (2^{1-2s}-1)\sumstar_{\substack{k_1 \\ (k_1, 1+i) = 1}}
   \sum_{k_2}f(k_1k^2_2) +\sumstar_{\substack{k_1 \\ (k_1, 1+i) \neq 1}}\sum_{k_2}f(k_1k^2_2).
\end{align*}

    We apply the above expression to recast $M'$ as
\begin{align}
\label{M'}
    M' &=\frac {\pi}{2}X \left (  \ \sumstar_{\substack{k_1 \\ (k_1, 1+i) = 1}} \mathcal{M}_{1}(s,u, k_1)+ \sumstar_{\substack{k_1 \\ (k_1, 1+i) \neq 1}}\mathcal{M}_{2}(s,u, k_1)\right ),
\end{align}
    where
\begin{align}
\label{eq:mathcalM1}
\begin{split}
    & \mathcal{M}_{1}(s,u, k_1) \\
    =& \left( \frac 1{2\pi i} \right)^2\int\limits_{(3/2)}\int\limits_{(\varepsilon)}\widehat{W}\left( 1-\frac s2 \right)\left( XN(k_1) \right)^{-s/2}(2^{1-2s}-1)\left(\frac {\pi}{\sqrt{2}} \right)^{-s}\frac{\Gamma (\frac s{2})}{\Gamma (\frac {1-s}{2})}\widehat{\Phi}(u)Y^u  J_{k_1} \left(\frac s2, u-\frac s2 \right) \dif u \dif s \\
    =& \left( \frac 1{2\pi i} \right)^2\int\limits_{(3/2)}\int\limits_{(3/4+\varepsilon)}\widehat{W} \left( 1-\frac s2 \right) \left( XN(k_1) \right)^{-s/2}(2^{1-2s}-1)\left(\frac {\pi}{\sqrt{2}}  \right)^{-s}
    \frac{\Gamma (\frac s{2})}{\Gamma (\frac {1-s}{2})}\widehat{\Phi} \left(u+\frac s2 \right)Y^{u+s/2} J_{k_1}(\frac s2, u) \dif u \dif s .
\end{split}
\end{align}
   and $J_{k_1}(v,w)$ is defined in \eqref{eq:J}.
   The formula for $\mathcal{M}_{2}(s,u,k_1,l)$ is identical to \eqref{eq:mathcalM1} except the factor $2^{1-2s} - 1$ is removed. \newline

    To evaluate  $\mathcal{M}_{1}(s,u, k_1)$ and $\mathcal{M}_{2}(s,u, k_1)$, we apply Lemma \ref{lemma:J} to replace $J_{k_1}(s, u)$ by $L(1/2+u, \chi_{i k_1})J_{2, k_1}(s,u)$ and
    shift the line of integration over $u$ to the line $\Re(u)=\varepsilon$. We encounter a pole of the Dirichlet $L$-functions at $u =1/2$
    for $k_1 = \pm i^3$ only. We denote the sum of the possible residues by $R$ and the remaining integrals by $I$. \newline

    To treat $R$, we move the line of integration over $s$ to the line $\Re(s)=2+\varepsilon$. We observe that as $\Phi$ and $W$ are compactly supported on $\mr^{+}$, integrating by parts implies that for any integers $E_i \geq 0$, $1 \leq i \leq 2$,
\begin{equation} \label{eq:h}
\left| \widehat {W}(s)\widehat {\Phi}(u) \right| \ll \frac{1}{|us| (1+|s|)^{E_1} (1+|u|)^{E_2}}.
\end{equation}

   We further deduce from Stirling's formula (see \cite[(5.112)]{iwakow}) that
\begin{align}
\label{gammabound}
\frac{\Gamma (s/2)}{\Gamma ((1-s)/2)} \ll |s|^{\Re(s)-1/2}.
\end{align}

    Combining \eqref{eq:h}, \eqref{gammabound} together with
    the estimation that $J_{2, \pm i^3}(s,u) \ll 1$ by Lemma \ref{lemma:J}, we estimate the remaining integral over $s$ on the line $\Re(s) = 2+\varepsilon$ to arrive at
\begin{align} 
\label{Rremainder}
 R \ll Y^{1/2} \left( \frac YX \right)^{1+\varepsilon}.
\end{align}

Now, to estimate the contribution of $I$ to \eqref{M'}, we move
the line of integration over $s$ to $\Re(s)=c=4+2\varepsilon$.
We find by Lemma \ref{lemma:J} that for any $\varepsilon>0$,
$$
J_{k_1}(s,u)
\ll (N(k_1))^{\varepsilon} \left| L\left( \frac{1}{2}+u, \chi_{i k_1} \right) \right|.
$$

Using \eqref{eq:h} with $E_1=E_2=1$ and \eqref{gammabound}, together with the above bound, 
we find that
\begin{align}
\label{eq:firstbd}
I \ll Y^{\varepsilon} \left(\frac YX \right)^{c/2}   \int\limits_{(c)}\int\limits_{(\varepsilon)}\sumstar_{N(k_1)}\frac{1}{N(k_1)^{c/2-\varepsilon}}  \left| L \left( \frac{1}{2}+u, \chi_{i k_1} \right) \right| \frac{ |s|^{\Re(s)-1/2} \ \dif u \,  \dif s}{|1-s/2|(1+|1-s/2|)|u+s/2|(1+|u+s/2|)}.
\end{align}

Now Lemma \ref{lem:HB} and partial summation give the bound
\begin{equation} \label{Lbound}
 \sumstar_{N(k_1)}\frac{1}{N(k_1)^{c/2-\varepsilon}} \left| L \left( \frac 12+u, \chi_{i k_1} \right) \right| \ll (1+|\Im(u)|)^{1/2+\varepsilon} \ll \left ((1+|u+s|)^{1/2+\varepsilon}+|s|^{1/2+\varepsilon} \right ).
\end{equation}

Inserting \eqref{Lbound} in \eqref{eq:firstbd} renders the bound
\begin{align}
\label{eq:firstbd1}
 I \ll Y^{\varepsilon} \left( \frac YX \right)^{c/2} \ll Y^{\varepsilon} \left( \frac YX \right)^{2+\varepsilon}. 
\end{align}

  As $X \geq Y$, we deduce from \eqref{M'}, \eqref{Rremainder} and \eqref{eq:firstbd1} that 
\begin{align}
\label{I2}
 M' \ll XY^{1/2}\left( \frac YX \right)^{1+\varepsilon}.
\end{align}

\subsection{Conclusion}
 Combining \eqref{M0} and \eqref{I2} gives that if $X \geq Y$, then
\begin{align*}
  S_2(X,Y; \Phi, W)=\frac {\pi XY^{1/2}}{24\zeta_{K}(2)}\widetilde{W}_i(0)\Phi \left( \frac 12 \right)+O \left( XY^{\theta/2}+XY^{1/2} \left( \frac YX \right)^{1+\varepsilon} \right).
\end{align*}
   The assertion of Theorem \ref{quadraticmean} now follows from the above expression upon setting
\begin{align}
\label{C}
   C(\Phi, W)=\frac {\pi}{24\zeta_{K}(2)}\widetilde{W}_i(0)\Phi \left( \frac 12 \right).
\end{align}

\vspace{0.1in}

\noindent{\bf Acknowledgments.} P. G. is supported in part by NSFC grant 11871082 and L. Z. by the FRG Grant PS43707 and the Goldstar Award PS53450 from the University of New South Wales (UNSW). Parts of this work were done when P. G. visited the School of Mathematics and Statistics at UNSW in September 2019.  He wishes to thank the school for the invitation, financial support and warm hospitality during his pleasant stay.  Also, the authors thank the anonymous referee for his/her
very careful reading of this manuscript and many helpful comments and suggestions.

\bibliography{biblio}
\bibliographystyle{amsxport}

\vspace*{.5cm}

\noindent\begin{tabular}{p{8cm}p{8cm}}
School of Mathematical Sciences & School of Mathematics and Statistics \\
Beihang University & University of New South Wales \\
Beijing 100191 China & Sydney NSW 2052 Australia \\
Email: {\tt penggao@buaa.edu.cn} & Email: {\tt l.zhao@unsw.edu.au} \\
\end{tabular}

\end{document}